\documentclass[imslayout,preprint]{imsart}

\usepackage[english]{babel}

\usepackage[pdftex,colorlinks,%
           bookmarks=true,%
           pdftitle=Clustring\ and\ percolation\ of\ point\ processes,%
           pdfauthor=B.\ Blaszczyszyn%
          \ -\  D.\ Yogeshwaran]{hyperref}

\usepackage{url} 

\usepackage{amsthm,amsmath,amssymb}

\RequirePackage[numbers]{natbib} 
\RequirePackage{hypernat}
\def\bdoi#1{\ignorespaces}  

\usepackage{graphicx}
\graphicspath{{./}{./figures/}{../subpoisson/figures/}{../R/}}

\RequirePackage[colorlinks,citecolor=blue,urlcolor=blue]{hyperref}

\startlocaldefs

\newcommand{\EXCLUDE}[1]{}
\newcommand{\be}{\begin{equation}}
\newcommand{\ee}{\end{equation}}
\newcommand{\bea}{\begin{eqnarray}}

\newcommand{\non}{\nonumber}
\newcommand{\eea}{\end{eqnarray}}
\newcommand{\sP}{\mathsf{P}}  
\newcommand{\pr}[1]{\mathsf{P}\left( #1 \right)}
\newcommand{\sE}{\mathsf{E}} 
\newcommand{\EXP}[1]{\mathsf{E}\!\left(#1\right) }

\newcommand{\1}[1]{\mathsf{1}\!\left[\,#1\,\right] }
\newcommand{\remove}[1]{}

\newtheorem{thm}{Theorem}[section]
\newtheorem{cor}[thm]{Corollary}
\newtheorem{lem}[thm]{Lemma}
\newtheorem{prop}[thm]{Proposition}
\newtheorem{remark}[thm]{Remark}
\newenvironment{rem}[1][]{\begin{remark}[#1]\rm}{\end{remark}}
\newtheorem{defin}[thm]{Definition}
\newenvironment{defn}[1][]{\begin{defin}[#1]\rm}{\end{defin}}
\newtheorem{example}[thm]{ Example}
\newenvironment{exe}[1][]{\begin{example}[#1]\rm}{\end{example}}

\newtheorem{Condition}[thm]{Condition}

\def\rar{\rightarrow}
\def\la{\langle}
\def\ra{\rangle}

\newcommand{\lam}{\lambda}
\newcommand{\Lam}{\Lambda}

\newcommand{\sg}{\sigma}

\def\mR{\mathbb{R}}
\def\mZ{\mathbb{Z}}

\def\mN{\mathbb{N}}

\def\mE{\mathbb{E}}
\def\mL{\mathbb{L}}

\newcommand{\ovr}{{\overline r}}

\newcommand{\md}{\text{d}}

\newcommand{\cX}{{\mathcal X}}

\newcommand{\cL}{\mathcal{L}}

\def\1{\mathbf{1}}

\endlocaldefs

\begin{document}

\begin{frontmatter}

\title{Clustering and percolation of point processes}
\runtitle{Clustering and  percolation of point processes}

\begin{aug}

\author{\fnms{Bart{\l}omiej}
\snm{B{\l}aszczyszyn}\corref{}\thanksref{inria,ens}\ead[label=e1]{Bartek.Blaszczyszyn@ens.fr}} 
\address{Inria/ENS\\
23 av. d'Italie\\
CS 81321\\
75214 Paris Cedex 13\\
France\\
\printead*{e1}}
\affiliation{Inria\thanksmark{inria} and ENS Paris\thanksmark{ens}}

and \author{\fnms{D.}
\snm{Yogeshwaran}\thanksref{Technion}\ead[label=e2]{yogesh@ee.technion.ac.il}\thanksref{t1}}
\thankstext{t1}{DY is grateful to the support of Inria  and ENS Paris where most of
this work was done. DY is thankful to research grants from EADS(Paris), Israel Science Foundation (No: 853/10) and AFOSR (No: FA8655-11-1-3039).}    
\address{Dept. of Electrical Engineering\\
Technion - Israel Institute of Technology\\
Haifa 32000\\
Israel\\
\printead*{e2}}
\affiliation{Technion-Israel Institute of Technology\thanksmark{Technion}, Haifa.}

\runauthor{B. B{\l}aszczyszyn and D. Yogeshwaran}

\end{aug}

\begin{abstract}\
\pdfbookmark[1]{Titlepage}{title} 
We are interested in phase transitions in certain percolation models
on point processes and their dependence on clustering properties of
the point processes. We show that point processes with smaller void probabilities and factorial moment
measures than the stationary Poisson point process exhibit
non-trivial phase transition in the percolation 
of some coverage models based on level-sets of additive
functionals of the point  process. Examples of such point processes
are determinantal point processes, some perturbed lattices,
and more generally, negatively associated point processes.  
Examples of such coverage models are $k$-coverage in the Boolean model
(coverage by at least $k$ grains) 
and SINR-coverage (coverage if
the signal-to-interference-and-noise ratio is large).
In particular, we answer in affirmative the hypothesis 
of existence of phase transition
in the percolation of $k$-faces in the \u{C}ech simplicial complex 
(called also clique percolation) on point processes which cluster less
than the Poisson process.
We also construct a Cox point process, which is  "more clustered"
than the Poisson point process and  whose Boolean model percolates for
arbitrarily small radius. This shows that  clustering (at least, as
detected by our specific tools) does not
always ``worsen'' percolation, as well as that 
upper-bounding this clustering  by a Poisson process 
is a consequential assumption for the phase transition to hold.
\end{abstract}
\begin{keyword}[class=AMS]
\kwd[Primary ]{60G55, 
               82B43, 
               60E15
}
\kwd[; secondary ]{60K35, 
                   60D05, 
                   60G60
}
\end{keyword}
\begin{keyword}
\kwd{point process}
\kwd{Boolean model}
\kwd{percolation}
\kwd{phase transition}
\kwd{shot-noise fields}
\kwd{level-sets}
\kwd{directionally convex ordering}
\kwd{perturbed lattices}
\kwd{determinantal}
\kwd{permanental point processes}
\kwd{sub-(super-)Poisson point processes.}
\end{keyword}
\end{frontmatter}
\clearpage 
\section{Introduction}
\label{sec:intro}

Starting with the work of \cite{Gilbert61}, percolation problems on geometric models defined over the Poisson point process have garnered interest among both stochastic geometers and
network theorists. Relying and building upon the (on-going) successful
study of its discrete counterpart(\cite{Grimmett99}), continuum
percolation has also received considerable attention as evinced in the
monographs \cite{MeeRoy96} and \cite{Penrose03}. Our work shall
deviate from the standard approach of studying geometric models based
on the Poisson point process, by focusing on percolation models
defined over general stationary point processes. In particular, we
shall try to {\em formalize the comparison of ``clustering'' phenomena in point processes  and investigate its impact on percolation 
models.}

The new approach contributes to both
theory and applications. Results regarding percolation  on "clustered"
or "repulsive" point  processes are scarce (we shall
say more in Section~\ref{sec:related}). Our methods help in
a more systematic study of models over general stationary point
processes apart from the ubiquitous Poisson point process. 
In particular, we {\em relate percolation properites of the
geometric models to the more intrinsic properties of point processes
such as moment measures and void probabilities}.
As regards applications, we observe that the Poissonian assumption on point
processes is not always preferable in many models of spatial
networks. In such scenarios, the question of whether clustering or
repulsion increases various performance measures (related to some
property of the underlying geometric graph) or not arises
naturally. Percolation is one such performance measure and we refer
the reader to the upcoming survey \cite[Section 4]{Sollerhaus12} for a
more comprehensive overview of applications of our methods for
comparison of various other properties of geometric graphs. 

With this succinct introduction, we shall now get into our key
percolation model. As a sample of our results, we will present
here in detail the particular case of $k$-percolation alone with a
brief discussion of other results. The point process notions used in the rest of the section
are described formally in Section \ref{s.point_processes}. 

\subsection{$k$-percolation}
\label{sec:results} 
By {\em percolation of a set} (in 
the Euclidean space usually), we mean that the set contains an
unbounded connected subset. We shall also use {\em percolation of a
  graph}, where it means existence of an infinite connected
subgraph. By $k$-percolation in a Boolean model, we understand
percolation of the subset of the space covered by at least $k$ grains
of the Boolean model and more rigorously, we define as follows.  
\begin{defn}[$k$-percolation]
\label{d.k-percol}
Let $\Phi$ be a point process in $\mR^d$, the $d$-dimensional Euclidean space. For $r \geq 0$ and $k \geq 1,$ we define the coverage number field $V_{\Phi,r}(y) :=  \sum_{X_i \in \Phi} \1[y \in B_r(X_i)]$,  where $B_r(x)$ denotes the Euclidean ball of radius $r$ centred at $x$. The $k$-covered set is defined as 
$$C_k(\Phi,r) := \{y : V_{\Phi,r}(y) \geq k \}.$$
Define the {\em critical radius for $k$-percolation} as 
$$ r^k_c(\Phi) := \inf \{r : \pr{\mbox{$C_k(\Phi,r)$ percolates}} > 0 \}\,,$$ 
where, as before, percolation means existence of an unbounded connected subset.
\end{defn}
Note that $C(\Phi,r):= C_1(\Phi,r)$ is the standard Boolean model or
continuum percolation model mentioned in the first paragraph. It can
also be viewed as a graph by taking $\Phi$ as the vertex set and the
edge-set being $E(\Phi,r) := \{(X,Y) \in \Phi^2 : 0 < |X - Y| \leq 2r
\}.$ This is the usual random geometric graph,
called also the Gilbert's disk graph.
 The two notions of
percolation (graph-theoretical and topological) are the same in this
case. The more general set $C_k(\Phi,r)$ can also be viewed as a graph
on $k$-faces of the \v{C}ech complex on $\Phi$ and this shall be
explained later in Remark \ref{rem:cech_complex}. Clearly,
$r_c(\Phi):=r_c^1(\Phi)$ is the {\em critical radius}
of the ``usual'' continuum percolation model on~$\Phi$, and we have 
$r_c(\Phi) \leq r_c^k(\Phi)$.  
We shall use clustering properties of $\Phi$ to get bounds on $r_c^k(\Phi)$.

\subsection{Clustering and percolation --- heuristics}
Before stating our percolation result, 
let us discuss more about clustering and percolation heuristics. Clustering of $\Phi$ roughly means that the points of $\Phi$ lie in clusters (groups) with the clusters being well spaced out. When trying to find the minimal $r$ for which the continuum percolation model $C(\Phi,r)$ percolates, we observe that points lying in the same cluster of $\Phi$ will be connected by edges for some smaller $r$ but points in different clusters need a relatively higher $r$ for having edges between them. Moreover, percolation cannot be achieved without edges between some points of different clusters. It seems to be evident that spreading points from clusters of $\Phi$ ``more homogeneously'' in the space would result in a decrease of the radius $r$ for which the percolation takes place. This is a heuristic explanation why clustering in a point process 
$\Phi$ should increase the critical radius $r_c(\Phi)$ and a similar
reasoning can also be given for $r^k_c(\Phi)$. Our study was motivated by
this heuristic.  

To make a formal statement of the above heuristic, one needs to adopt
a tool to compare clustering properties of point
processes\footnote{The reader should also keep in mind that we are
  interested in looking at point processes of same intensities and
  hence the usual strong order is not a suitable measure.}. In this
regard, our initial choice was  {\em directionally convex} ($dcx$)
order \footnote{The $dcx$ order of random vectors is an integral order
  generated by twice differentiable functions with all their second
  order partial derivatives being non-negative. Its extension  to
  point processes consists in comparison of vectors of number of
  points in every possible finite collection of bounded Borel subsets
  of the space.} (to be formally defined later). It has its roots in
\cite{snorder}, where one shows various results as well as examples
indicating that the $dcx$ order on point processes implies ordering of
several well-known clustering characteristics in spatial statistics
such as Ripley's K-function and  
second moment densities. Namely, a point process that is larger in the
$dcx$ order exhibits more clustering, while having equal mean number
of points in any given set. 

In view of what has been said above, a possible formalization of the
heuristic developed above, that clustering ``worsens'' percolation, 
would be $\Phi_1 \leq_{dcx} \Phi_2$ implies $r_c(\Phi_1)
\leq r_c(\Phi_2)$. The numerical evidences gathered for a certain
class of point processes, called perturbed lattice point processes,
were  supportive of this conjecture (cf~\cite{dcx-clust}). However,
 such a statement is not true in full
generality and we shall provide a counterexample in Section~\ref{sec:ex}.

We can also consider some weaker notions of clustering, for which only moment measures or void
probabilities can be compared. 
\begin{defn}
\label{defn:weak_subP} 
A point process $\Phi$~is said to be {\em weakly sub-Poisson}
if the following two conditions are satisfied:
\begin{eqnarray}
\pr{\Phi(B)=0} & \le & e^{-\EXP{\Phi(B)}} \, \, \, \, \, \, \, \, \, \, \, \, \, \, \, \mbox{($\nu$-weakly sub-Poisson)} \label{e.nu-weakly-sub-Poisson}  \\
\EXP{\prod_{i=1}^k\Phi(B_i)} & \le & \prod_{i=1}^k\EXP{\Phi(B_i)} \, \, \, \, \, \, \, \mbox{($\alpha$-weakly sub-Poisson)} \label{e.alpha-weakly-sub-Poisson}  
\end{eqnarray}
where $B_i\subset\mR^d$ are mutually disjoint bBs and $B$ is any bBs. If only either of the conditions is satisfied, accordingly we call the point process
to be $\nu$-weakly sub-Poisson ($\nu$ stands for void probabilities) or $\alpha$-weakly sub-Poisson ($\alpha$ stands for moment measures).
Similar notions of super-Poissonianity can be defined by reversing the inequalities. 
\end{defn}

Again, larger  values of these characteristics (i.e,void probabilities and factorial moment measures) suggest more
clustering and hence they  also can be used to compare  clustering of point processes.
This comparison is weaker than $dcx$ order and association as well. More details on these notions are
provided in Section \ref{s.Notation}. 

It is explained in~\cite{dcx-clust}, that weak sub-Poissonianity allows
for a slightly different conclusion than the heuristic described before the definition. Suppose that $\Phi_2$ is a stationary Poisson point process
and $\Phi_1$ is a stationary weakly sub-Poisson point process. Then, $\Phi_1$ exhibits the usual 
phase transition $0<r_c(\Phi_1)<\infty$, provided $\Phi_2$
exhibits a (potentially) stronger, ``double phase transition'': 
$0<\underline r_c(\Phi_2)$ and
$\overline r_c(\Phi_2)<\infty$, where $\underline r_c(\Phi_2)$ and
$\overline r_c(\Phi_2)$ are some ``nonstandard'' critical radii~\footnote{%
$\underline r_c(\Phi_2)$ and
$\overline r_c(\Phi_2)$ are critical radii related, respectively, 
to the finiteness of asymptotic of the expected number
of long occupied paths from the origin and the expected number of
void circuits around the origin in suitable discrete
approximations of the continuum model.}
sandwiching $r_c(\Phi_2)$, and exhibiting  opposite monotonicity with
respect to clustering:
$$\underline r_c(\Phi_2)\le \underline r_c(\Phi_1)\le r_c(\Phi_1)
\le \overline r_c(\Phi_1)\le \overline r_c(\Phi_2)\,.$$
Conjecturing that the above ``double phase transition'' 
holds for the Poisson point process $\Phi_2$, one obtains the  result  on
the ``usual'' phase transition for  $r(\Phi_1)$  for all weakly sub-Poisson point processes $\Phi_1$.
From \cite[Proposition 6.1]{dcx-clust}, we know that $\underline
r_c(\Phi_2) \geq (\theta_d)^{-\frac{1}{d}}$, where $\theta_d$ is the
volume of $d$-dimensional Euclidean unit ball. However,  finiteness of $\ovr_c(\Phi_2)$ is not clear and hence, in this paper we will prove the results in a slightly different way. The generality of this method shall be explained below and will be obvious from the results in Section \ref{sec:sub-poisson}.

\subsection{Results}
We can finally state now one of the main results on phase transition in
$k$-percolation which follows as a corollary from Theorem~\ref{thm:k_perc_sub-Poisson_pp} (to be stated later). Recall the definition of weak sub-Poissonianity 
from Definition \ref{defn:weak_subP}. 
\begin{cor}
\label{cor:non_trivial_k_perc} 
Let $\Phi$ be a simple, stationary, weakly sub-Poisson point process of intensity $\lambda$. For $k
\geq 1$, $\lam > 0$, there exist constants $c(\lam)$ and  
$c(\lam,k)$ (not depending on the distribution of $\Phi$) such that
$$ 0 < c(\lam) \leq r_c^1(\Phi) \leq r_c^k(\Phi) \leq c(\lam,k) < \infty. $$
\end{cor}
Examples of stationary weakly sub-Poisson processes 
are {\em perturbed latices with convexly sub-Poisson
replication kernels} and {\em determinantal point processes with trace-class integral kernels}
(cf~\cite{dcx-clust}).  

\paragraph{Further results}
Random field $V_{\Phi,r}(\cdot)$ introduced in
Definition~\ref{d.k-percol} is but one example of {\em an additive
  shot-noise field} (to be defined in Section~\ref{ss:bounds_shot-noise}). 
Replacing the indicator
function in the definition of $V_{\Phi,r}(\cdot)$ by more general
functions of $X_i$, one can get 
more general additive shot-noise fields. $C_k(\Phi,r)$ is an example of an excursion set or
a level-set (i.e, sets of the form $\{ V_{\Phi,r}(\cdot) \geq h \}$ or its complement for some $h \in \mR$) that can be associated with the random field
$V_{\Phi,r}(\cdot)$. In this paper  we shall
develop more general methods suitable for the study of 
percolation of level sets of additive
shot-noise fields. Besides $C_k(\Phi,r)$, we  shall
apply these methods to study percolation of SINR coverage model (coverage by
a signal-to-interference-and-noise ratio;
\cite{sinr_coverage,Dousse_etal_TON,Dousse_etal06}) in non-Poisson setting.
Though, we do not
discuss it here, our methods can also be used to show non-trivial
phase transition in the continuum analogue of word percolation (cf.\cite{Kesten01}). The details can be found in~\cite[Section~6.3.3]{Yogesh_thesis}. 

\subsection{Paper organization} 

The necessary notions, notations
as well as some preliminary results 
are introduced and recalled in Section~\ref{s.Notation}. 
In Section~\ref{sec:sub-poisson} we state and prove our main results
regarding the existence of the phase transition
for percolation models driven by sub-Poisson point processes.
A Cox point process, which is $dcx$ larger 
than the Poisson point process (clusters more) 
and  whose Boolean model percolates for
arbitrarily small radius ($r_c=0$) is provided in the Section~\ref{sec:ex}.

\subsection{Related work}
\label{sec:related}
Let us first remark on 
studies in continuum percolation which are comparisons of different models driven by the same (usually Poisson)
point process. In~\cite{Jonasson01}, it was shown that  the critical intensity for
percolation of the Poisson Boolean model on the plane is minimized when the shape of
the typical grain is a triangle and maximized when it is a centrally symmetric set. Similar result
was proved in~\cite{Roy02} using more probabilistic arguments
for the case when the shapes are taken over the set of all
polygons and the idea was also used for three dimensionial Poisson Boolean models. It is known for many
discrete graphs that bond percolation is strictly easier than site
percolation. A similar result as well as strict
inequalities for spread-out connections in the Poisson random connection model has been proved in~\cite{Franc_etal05,Franc_etal11}.

For determinantal point processes, \cite[Cor. 3.5]{GeorgiiYoo} shows non-existence of percolation for small enough integral kernels 
(or equivalently for small enough radii) via coupling with a Poisson
point process. This shows non-zero critical radius ($r_c>0$) for percolation 
of determinantal point processes. 

Non-trivial critical radius for continuum percolation on 
point processes representing zeros of Gaussian analytic functions,
is shown in \cite{Ghosh12}. These processes are reputed  to cluster
less than Poisson processes, however currently we are not able to make
this comparison formal using our tools.
Also, \cite{Ghosh12} shows uniqueness of infinite clusters 
for both zeros of Gaussian analytic functions and the Ginibre point
process (a special case of determinantal process).

Critical radius of the continuum percolation model on
the hexagonal lattice perturbed by the Brownian motion is studied in a recent pre-print~\cite{BenjaminiStauffer2011}.%
This is an example of our perturbed lattice and as such it is a $dcx$ sub-Poisson point process.%
\footnote{More precisely, at any time $t$ of the evolution of the Brownian motion, it is $dcx$ smaller than  a non-homogeneous
Poisson point process of some intensity which depends on~$t$, and converges to the homogeneous one for $t\to\infty$.}
It is shown that for a short enough time of the evolution of the Brownian motion the critical radius is not larger than that of the
non-perturbed lattice.  This result is shown by some coupling  in the sense of set inclusion of point processes.

Many other inequalities in percolation theory depend on such coupling arguments (cf. e.g.~\cite{Liggett97}), which  for obvious reasons are not suited to comparison of point
processes with the same mean measures. 

\section{Notions, notation and basic observations}
\label{s.Notation}
\subsection{Point processes}
\label{s.point_processes}
Let $\texttt{B}^d$ be the Borel $\sg$-algebra and $\texttt{B}_b^d$ be
the $\sg$-ring of {\em bounded (i.e., of compact
  closure) Borel subsets} (bBs) in the
$d$-dimensional Euclidean space $\mR^d$. Let $\mN^d =
\mathbb{N}(\mR^d)$ be the space of non-negative Radon (i.e., finite on
bounded sets) counting measures on $\mR^d$. The Borel $\sg$-algebra $\mathcal{N}^d$ is generated by the
mappings $\mu \mapsto \mu (B)$ for all $B$ bBs. A point process
$\Phi$ is a random element in $(\mN^d,\mathcal{N}^d)$ i.e, a
measurable map from a probability space $(\Omega,\mathcal{F},\sP)$ to
$(\mN^d,\mathcal{N}^d)$. Further, we shall say that a point process $\Phi$ is
simple if a.s. $\Phi(\{x\}) \leq 1$ for all $x \in \mR^d$. As always, a point process on $\mR^d$ is said to be {\em stationary} if its 
distribution is invariant with respect to translation by vectors in~$\mR^d$. 

The measure  $\alpha^{k}(\cdot)$ defined by $\alpha^{k}(B_1\times\ldots\times B_k) = \EXP{\prod_{i=1}^k \Phi(B_i)}$ for all (not
necessarily disjoint) bBs $B_i$  ($i=1,\ldots, k$) is called {\em $k\,$th order moment
measure} of $\Phi$. For simple point processes, the truncation of the measure $\alpha^k(\cdot)$ to the subset
$\{(x_1,\ldots,x_k)\in (\mR^{d})^k: x_i\not=x_j,\; \text{for\;} i\not=j\}$ is equal to the {\em $k\,$th order factorial moment
measure} $\alpha^{(k)}(\cdot)$. To explicitly denote the dependence on
$\Phi$, we shall sometimes write as $\alpha_{\Phi}^{k}$ for the moment
measures and similarly for factorial moment measures. This is the standard framweork for point processes and more generally, random measures
(see \cite{Kallenberg83}). 

We shall be very brief in our introduction to stochastic ordering of point processes necessary for our purposes. We shall start with a discussion on weakly sub-Poisson point processes
defined in Definition \ref{defn:weak_subP}. 

\subsection{Weak sub- and super-Poisson point processes}
\label{ss.sub-suuper-Poisson}
Note that, $\alpha$-weakly sub-Poisson (super-Poisson) point process have
factorial moment measures $\alpha^{(k)}(\cdot)$ 
smaller (respectively larger) than those of
the Poisson point process of the same mean measure;
inequalities hold everywhere provided  the point processes is simple
and ``off the diagonals'' otherwise.
Recall also that moment measures $\alpha^{k}(\cdot)$ of a general
point process can be expressed as non-negative 
combinations of its factorial moment measures 
(cf \cite[Ex.~5.4.5, p.~143]{DVJI2003}). Consequently,
simple, {\em $\alpha$-weakly sub- (super-)Poisson point processes 
have also  moment measures $\alpha^{k}(\cdot)$ 
smaller (larger) than those of Poisson point process.} 

The reason for using the adjective weak will be clear once we introduce the stronger notion of {\em directionally convex
ordering}. This will be very much needed for the example presented in Section \ref{sec:ex} but not for the results of Section \ref{sec:sub-poisson}. So, the reader may skip the following subsection now and return back to it later when needed. 

\subsection{Directionally convex ordering}
Let us quickly introduce the theory of directionally convex
ordering. We refer the reader to \cite[Section 3.12]{Muller02} for a
more detailed introduction.

For a function $f:\mR^k \rar \mR$, define the discrete differential operators as $\Delta_{\epsilon}^if(x) := f(x + \epsilon e_i) - f(x),$ where $\epsilon > 0, 1 \leq i \leq k$ and $\{e_i\}_{1 \leq i \leq k}$ are the canonical basis vectors for $\mR^k$. Now, one introduces the following families of  {\em Lebesgue-measurable} functions on $\mR^k$:
A function $f:\mR^k \rar \mR$ is said to be {\em directionally convex}~($dcx$) if for every $x \in \mR^k, \epsilon,\delta > 0, i,j \in \{1,\ldots,k\}$, we have that $\Delta_{\epsilon}^i \Delta_{\delta}^jf(x) \geq 0.$ We abbreviate {\em increasing} and $dcx$ by $idcx$ and {\em decreasing} and $dcx$ by $ddcx$. There are various equivalent definitions of these and other multivariate functions suitable for dependence ordering (see \cite[Chapter 3]{Muller02}).

Unless mentioned, when we state $\sE(f(X))$ for a function $f$ and a random vector $X$ , we assume that the expectation exists.
Assume that $X$ and $Y$ are real-valued random vectors of the
same dimension. Then  {\em $X$ is said to be less than $Y$ in
$dcx$ order} if $\sE(f(X)) \leq \sE(f(Y))$ for all $f$ $dcx$
such that both the expectations are finite. We shall
denote it as $X \leq_{dcx} Y$.
This property clearly regards only the distributions of $X$ and $Y$,
and hence sometimes we will say that the law of $X$ is less in
$dcx$ order than that of $Y$.

A point process $\Phi$ on $\mR^d$
can be viewed as the random
field $\{\Phi(B)\}_{B \in \texttt{B}_b^d}$. As the $dcx$ ordering for random fields is defined via comparison of
their finite dimensional marginals, for two
point processes on $\mR^d$, one says that  $\Phi_1(\cdot) \leq_{dcx} \Phi_2(\cdot)$, if for any $B_1, \ldots, B_k$ bBs in $\mR^d$,
\begin{equation}
\label{defn:dcx_rm}
(\Phi_1(B_1),\ldots,\Phi_1(B_k)) \leq_{dcx} (\Phi_2(B_1),\ldots,\Phi_2(B_k)).
\end{equation}
The definition is similar for other orders, i.e., those defined by $idcx,\allowbreak ddcx$ functions.
It was shown in~\cite{snorder} that it is enough to verify the above
condition for $B_i$ mutually disjoint.

In order to avoid technical difficulties, we will
consider here only point processes whose  {\em mean measures} $\sE(\Phi(\cdot))$
are Radon (finite on bounded sets). For such point processes, $dcx$ order is a
transitive order.
Note also that
$\Phi_1(\cdot) \leq_{dcx} \Phi_2(\cdot)$ implies the {\em equality
of their  mean measures}: $\sE(\Phi_1(\cdot)) =\sE(\Phi_2(\cdot))$ as both $x$ and $-x$ are $dcx$ functions on
$\mR$.
For more details on $dcx$ ordering
of point processes and random measures, see~\cite{snorder}.

\subsection{Examples}
\label{ss:sim_ex}
We now concentrate on comparison of point processes to the Poisson point process of same
mean measure. Following~\cite{dcx-clust}  we will call a point process
{\em$dcx$ sub-Poisson} (respectively {\em
 $dcx$  super-Poisson}) if it is smaller (larger) in $dcx$ order than the
Poisson point process (necessarily of the same mean measure).
For simplicity, we will just refer to them as sub-Poisson or super-Poisson point process
omitting the word $dcx$. 

From \cite[Proposition 3.1 and Fact 3.2]{dcx-clust}, we can see that weak sub- and super-Poissonianity are actually
weaker than that of $dcx$ sub- and super-Poissonianity respectively. Interestingly, they are also weaker than the notion
of association (see \cite[Section 2]{dcx-clust}). More precisely, it is shown in~\cite[Cor. 3.1]{dcx-clust} that under very mild regularity conditions, positively associated 
point processes are weakly super-Poisson, while negatively associated point processes are weakly sub-Poisson.

We list here briefly some examples of $dcx$ and weak sub-Poisson and super-Poisson point processes. It was observed in~\cite{snorder} that some doubly-stochastic Poisson
(Cox) point processes, such as Poisson-Poisson cluster point processes and, more generally, L\'{e}vy based Cox point processes are super-Poisson. \cite{BurtonWaymire1985} provide examples of positively associated Cox point processes, namely those driven by a positively associated random measure. 

A rich class of point processes called the perturbed lattices, including both sub- and super-Poisson point processes,  is provided in~\cite{dcx-clust} (see Section~\ref{sec:ex} for one of the simpler perturbed lattices). These point processes can be seen as toy models for determinantal and permanental point processes; cf.~\cite{Ben09}. Regarding these latter point processes, it is shown in~\cite{dcx-clust} that determinantal and permanental point processes are weakly sub-Poisson and weakly super-Poisson respectively.

\section{Non-trivial phase transition for percolation models on
 sub-Poisson point processes}
\label{sec:sub-poisson}


We will be particularly interested in  percolation models on
level-sets of additive shot-noise fields. The rough idea is as
follows: level-crossing probabilities for these models can be bounded
using Laplace transform of the underlying point process. For weakly
sub-Poisson point process, this can further be bounded by the Laplace
transform of the corresponding Poisson point process, which has a
closed-form expression. For 'nice' response functions of the
shot-noise, these expressions are amenable enough to deduce the
asymptotic bounds that are good enough to use the standard arguments
of percolation theory. Hence, we  can deduce percolation or
non-percolation of a suitable discrete approximation of the model. 
%
\remove{
\begin{Condition}
\label{Condn:DDCX_NA}
\begin{enumerate}
\item $\Phi \leq_{ddcx} \Phi_{\lam},$ where $\Phi_\lambda$ denotes the homogeneous Poisson point process with intensity $\lambda$ on $\mR^d$.
\item $\Phi$ is a negatively associated point process such that $\Phi(B) \leq_{ddcx} \Phi_{\lam}(B)$ for all $B$ bBs. 
\end{enumerate}
\end{Condition}  }
%


\subsection{Bounds on Shot-Noise fields}
\label{ss:bounds_shot-noise}

Denote by  
$$V_{\Phi}(y) := \sum_{X \in \Phi}\ell(X,y)$$
the (additive) shot-noise field generated by a point process $\Phi$ where $\ell(\cdot,\cdot) : \mR^d\times\mR^d \to \mR_+$ is called as the response function. The response function is assumed to be Lebesgue measurable in its first co-ordinate. We shall start with two complimentary results about ordering of Laplace transforms of weakly sub-Poisson (super-Poisson) point processes which are needed for the  Proposition \ref{prop:bounds_shot-noise} that follows them. This key proposition will provide us with bounds on level-crossing probabilities of shot-noise fields which drive all the proofs that follow later in the section. We shall prove the two lemmas only for the case of weakly sub-Poisson point processes and the analagous results for weakly super-Poisson point processes follow via similar arguments. 
\begin{lem}
\label{lem.moments_Laplace}
Assume that $\Phi$ is a simple point process of a Radon mean measure $\alpha$.
If $\Phi$ is $\alpha$-weakly sub-Poisson then for all $y \in \mR^d$,
\begin{equation}
\label{e.moments-Laplace}
\EXP{e^{V_{\Phi}(y)}} \le 
\exp\left[\int_{\mR^d}(e^{\ell(x,y)}-1)\,\alpha(\md x)\right] \,.
\end{equation}
If $\Phi$ is $\alpha$-weakly super-Poisson then the above inequality is reversed.     
\end{lem}
\begin{proof} 
From the known representation of the Laplace transform of a functional
of Poisson point process $\Phi_\alpha$ of intensity measure $\alpha$
(cf~\cite[eqn.~9.4.17 p.~60]{DVJII2007}), we
observe that the RHS of (\ref{e.moments-Laplace}) is the same as
$\EXP{e^{V_{\Phi_{\alpha}}(y)}}$. So, the rest of the proof will be
only concerned with proving that for any $y \in \mR^d$, 
\begin{equation}
\label{eqn:ineq1_shot-noise}
 \EXP{e^{V_{\Phi}(y)}} \leq  \EXP{e^{V_{\Phi_{\alpha}}(y)}}.
\end{equation}
From Taylor's series expansion for the exponential function and the positivity of summands, we get that
$$\EXP{e^{V_{\Phi}(y)}} =1+\sum_{k=1}^\infty
\int_{\mR^{dk}}\ell(x_1,y)\ldots\ell(x_k,y)\,\alpha_\Phi^{k}(\md(x_1,\ldots,x_k)\,,
$$
where $\alpha_\Phi^{k}$ are moment measures of $\Phi$.
By the assumption that $\Phi$ is simple and $\alpha$-weakly
sub-Poisson  $\alpha_\Phi^{k} \le \alpha_{\Phi_\alpha}^{k}$,
which completes the proof.
\end{proof}

\begin{lem}
\label{lem.voids_Laplace}
Assume that $\Phi$ is a simple point process of a Radon mean measure $\alpha$.
$\Phi$ is $\nu$-weakly sub-Poisson if and only if for all $y \in \mR^d$,
\begin{equation}
\label{e.void-Laplace}
\EXP{e^{-V_{\Phi}(y)}} \le 
\exp\left[\int_{\mR^d}(e^{-\ell(x,y)}-1)\,\alpha(\md x)\right] \,.
\end{equation}
$\Phi$ is $\nu$-weakly super-Poisson if and only if the above inequality is reversed.   
\end{lem}
\begin{proof}
Firstly, let us prove the easy implication by assuming that (\ref{e.void-Laplace}) holds. Let $B$ be a bBS and set $\ell(x,y) = t\1[x \in B]$ for $t > 0$. Then $V_{\Phi}(y) = t\Phi(B)$ and so we get the required inequality (\ref{e.nu-weakly-sub-Poisson}) to prove $\nu$-weak sub-Poissonianity of $\Phi$ :
$$\nu(B) = \lim_{t \to \infty}\EXP{e^{-t\Phi(B)}} \leq \lim_{t \to \infty} \exp\left[(e^{-t} -1) \alpha(B)\right] = e^{-\alpha(B)},$$
where the inequality is due to (\ref{e.void-Laplace}) for our specific choice of $\ell$.

Now for the reverse implication, assume that $\Phi$ is $\nu$-weakly sub-Poisson. As with many other proofs, we shall only prove the inequality in the case of simple functions i.e, $\ell(\cdot,y) = \sum_{i=1}^k t_i \1[x \in B_i]$ for disjoint bBs $B_i, i=1,\ldots,k$ and appeal to standard-measure theoretic arguments for extension to the general case. Thus for a simple function $\ell(\cdot,y)$, we need to prove the following: 
\begin{equation}
\label{e.void-simple-Laplace}
\EXP{e^{-\sum_i t_i\Phi(B_i)}}\le 
\prod_i \exp\left[\alpha(B_i)(e^{-t_i}-1)\right]\,.
\end{equation}
Setting $s_i = e^{-t_i}$, let $\Phi'$ be the 
thinned point process obtained from $\Phi$ by deleting points
independently with probability $p(x)$, where  $p(x) = s_i$ for  points
$x \in B_i$ and  $p(x) \equiv 1$ outside $\bigcup_i B_i$. Similary, we define $\Phi'_{\alpha}$ for the Poisson point process $\Phi_{\alpha}$ of intensity measure $\alpha$. Thus, we have that
$$\EXP{\prod_i s_i^{\Phi(B_i)}} = \sum_{n_1,\ldots,n_k \geq 0}\prod_i s_i^{n_i} \pr{\Phi(B_i) = n_i, 1 \leq i \leq k} = \pr{\Phi'(\bigcup_i B_i)=0}.$$
Now to prove (\ref{e.void-simple-Laplace}), it suffices to prove that for any bBs $B$,
$$ \pr{\Phi'(B)=0} \leq \pr{\Phi'_{\alpha}(B)=0} .$$
This follows from a more general observation that independent thinning preserves ordering of void probabilities of simple point processes, which we show in the remaining part of the proof using
a coupling argument.

Consider a  null-array of partitions $\{B_{n,j} \}_{n \geq 1, j \geq 1}$ of $\mR^d$  
\footnote{i.e., $\{B_{n,j}\}_{j  \geq 1}$ form a finite partition of $\mR^d$ 
for every $n$ and $\max_{j \geq 1}\{|B_{n,j}|\} \to 0$ as $n \to \infty$ where
$|\cdot|$ denotes the diameter in any fixed metric;
see \cite[page~11]{Kallenberg83})}. Further, we assume that either $B_{n,j} \subset B_i$ for some $i \in \{1,\ldots,k\}$ or $B_{n,j} \cap (\cup_i B_i) = \emptyset$. Such a choice of partition can be always made by refining any given partition $\{B'_{n,j} \}_{n \geq 1, j \geq 1}$ to $\{B'_{n,j} \cap B_i \}_{i \in \{1,\ldots,k\}, n \geq 1, j \geq 1}$. For every $x \in \mR^d$, let
$j(n,x)$ be  the unique index such that $x \in B_{n,j(n,x)}$. 
For every $n,j$, define $s(n,j) = s_i$ if $B_{n,j} \subset B_i$ else $s(n,j) \equiv 1.$ Thus by the choice of partition, we get that $s(n,j(n,x)) = s_i$ if $x \in B_i$ for some $i \in \{1,\ldots,k\}$ or else $s(n,j(n,x)) = 1$. Let $\{\xi\}=\{\xi_{n,j}\}_{n \geq 1, j \geq 1}$ be a family of independent
Bernoulli random variables $\text{Ber}(1-s(n,j))$ defined on a common probability
space with $\Phi$ and independent of it.
Define the family of  point processes  $\Phi_n = \sum_{X_i \in \Phi}
\xi_{n,j(n,X_i)}\delta_{X_i}$. For given $n\ge1$, $\Phi_n$
is a (possibly dependent) thinning of $\Phi$. Moreover,  because
$\Phi$ is simple, $\Phi_n(B)$ converges 
in distribution to $\Phi'(B)$ (recall, $\Phi'$  is an independent
thinning of $\Phi$ with retention probability $1-s_i$ in $B_i$ and 0
outside $\bigcup_i B_i$). The result follows by conditioning 
on $\{\xi\}$: 
$$\nu'(B)=\lim_{n\to\infty}\pr{\Phi_n(B)=0} = \lim_{n\to\infty}\EXP{\pr{\Phi(B^1_n)=0\,|\,(\xi)}} = \lim_{n\to\infty}\EXP{\nu(B^1_n)},$$
with  $B_n^1=\bigcup_{j:\xi_{n,j}=1}B_{n,j}$. This completes the proof as $\nu(B^1_n) \leq \pr{\Phi_{\alpha}(B_n^1)}$ for all realizations of $\xi$.
\end{proof}

\begin{prop}
\label{prop:bounds_shot-noise}
Let $\Phi$ be a simple, stationary point
process of  intensity $\lam$.  
\begin{enumerate}
\item If $\Phi$ is $\alpha$-weakly sub-Poisson then we have that for
  any $y_1,\ldots,y_m \in \mR^d$ and $s,h > 0$,  
\begin{equation}
\label{eqn:upper_bd_shot-noise2} \pr{V_{\Phi}(y_i) \geq h, 1 \leq i \leq m} \leq e^{-smh} \exp \left\{\lam \int_{\mR^d}(e^{s\sum_{i=1}^m\ell(x,y_i)} -1) \md x\right\}.
\end{equation}
\item If $\Phi$ is $\nu$-weakly sub-Poisson,  then we have that for any
  $y_1,\ldots,y_m \in \mR^d$ and $s,h > 0$,  
\begin{equation}
\label{eqn:lower_bd_shot-noise2} \pr{V_{\Phi}(y_i) \leq h, 1 \leq i \leq m} \leq e^{smh} \exp \left\{\lam \int_{\mR^d}(e^{-s\sum_{i=1}^m\ell(x,y_i)} -1) \md x\right\}.
\end{equation}
\end{enumerate}
\end{prop}
\begin{proof}
Suppose that inequalities (\ref{eqn:upper_bd_shot-noise2}) and
(\ref{eqn:lower_bd_shot-noise2}) are true for $m =1$. Observe that
$\sum_{i=1}^mV_{\Phi}(y_i) = \sum_{X \in \Phi}\sum_{i=1}^m\ell(X,y_i)$
is itself a shot-noise field driven by the response function
$\sum_{i=1}^m\ell(\cdot,y_i).$ Thus if (\ref{eqn:upper_bd_shot-noise2})
and (\ref{eqn:lower_bd_shot-noise2}) are true for $m =1$, we can get
the general case from the following easy inequalities:  
\begin{eqnarray*}
\pr{V_{\Phi}(y_i) \geq h, 1 \leq i \leq m} & \leq & \pr{\sum_{i=1}^mV_{\Phi}(y_i) \geq mh} \\
\pr{V_{\Phi}(y_i) \leq h, 1 \leq i \leq m} & \leq & \pr{\sum_{i=1}^mV_{\Phi}(y_i) \leq mh}.
\end{eqnarray*}

Hence, we shall now prove only the case of $m = 1$ for both the inequalities. Setting $y_1 = y$ and using Chernoff's bound, we have that
\begin{eqnarray*}
\pr{V_{\Phi}(y) \geq h} & \leq & e^{-sh}\EXP{e^{sV_{\Phi}(y)}}, \\
\pr{V_{\Phi}(y) \leq h} & \leq & e^{sh}\EXP{e^{-sV_{\Phi}(y)}}.
\end{eqnarray*}
Using Lemmas~\ref{lem.moments_Laplace} and \ref{lem.voids_Laplace}, we can upper bound the RHS of the both the equations and thus we have shown the inequalities in the case $m = 1$ as required. 
\end{proof}
\begin{rem}
\label{rem:two_bounds}
Assuming in Proposition~\ref{prop:bounds_shot-noise} 
that $\Phi$ is stationary and weakly sub-Poisson, we compare $\Phi$ 
to the Poisson point process of the same intensity ($\lambda$).
We see from the proof that this assumption can be weakened in the
following way: 
\begin{enumerate}
\item If $\Phi$ is simple and its factorial moment measures 
$\alpha_\Phi^{(k)}$  can be bounded by those of a homogeneous Poisson point
process of some intensity $\lambda'>0$  then (\ref{eqn:upper_bd_shot-noise2})
holds true for any
  $y_1,\ldots,y_m \in \mR^d$ and $s,h > 0$ with $\lambda$ replace by
  $\lambda'$. 
\item If $\Phi$ is simple and its void probabilities 
$\nu(\cdot)$  can be bounded by those of a homogeneous Poisson point
process of some intensity $\lambda''>0$  then~(\ref{eqn:lower_bd_shot-noise2})
holds true for any $y_1,\ldots,y_m \in \mR^d$ and $s,h > 0$ with with $\lambda$ replace by  $\lambda''$. 
\end{enumerate}
The two bounds can be obtained by comparison to Poisson processes of
two different intensities, with $\lambda''\le \lambda'$.
\end{rem}

\subsection{Auxiliary discrete models}
\label{ss.AuxModels}
Though we focus on the percolation of Boolean models (continuum percolation models),
but as is the wont in the subject we shall extensively use discrete
percolation models as approximations.
For $r > 0, x \in \mR^d$, define the following subsets of $\mR^d$. Let
$Q_r :=  (-r,r]^d$ and $Q_r(x) := x + Q_r$.
We will consider the following discrete graph : 
$\mL^{*d}(r) = (r\mZ^d,\mE^{*d}(r))$ is a
close-packed graph on the scaled-up lattice $r\mZ^d$; the edge-set is
$\mE^{*d}(r) := \{ \langle z_i,z_j \rangle \in (r\mZ^d)^2 : Q_r(z_i) \cap
Q_r(z_j) \neq \emptyset \}$.

In what follows, we will define auxiliary site percolation models
on the above graph by randomly declaring some of its vertices (called also
sites) open.  As usual, we will say that a given discrete
site percolation model percolates if the corresponding sub-graph consisting of all open sites contains an infinite component. 

Define the corresponding lower and upper level sets of the
shot-noise field $V_{\Phi}(.,.)$ on  the lattice $r\mZ^d$ by
$\mZ^d_r(V_{\Phi},\leq h) := \{z \in r\mZ^d : V_{\Phi}(z) \leq h \}$
and $\mZ^d_r(V_{\Phi},\geq h) := \{z \in r\mZ^d : V_{\Phi}(z) \geq h
\}$. The percolation of these two discrete models (i.e,
$\mZ^d_r(V_{\Phi},\leq h)$ and $\mZ^d_r(V_{\Phi},\geq
h)$) understood in the sense of site-percolation of the close-packed
lattice $\mL^{*d}(r)$ (cf Section~\ref{ss.AuxModels}) will be of interest to us.

There are two standard arguments used in percolation theory to show
non-percolation and percolation. We shall describe them here below as
we use one or the other of these two arguments in our proofs for
Theorems \ref{thm:k_perc_sub-Poisson_pp}, \ref{thm:Peierls_arg_bd_pp}, \ref{thm:sinr_poisson_perc} and \ref{thm:perc_sinr_sub-Poisson}.

\begin{rem}[Standard argument for non-percolation]
\label{rem:non_perc_standard_arguments}
Since the number of paths on $\mL^{*d}(r)$ of length $n$ starting from the
origin is at most $(3^d-1)^n$, in order to show non-percolation of a
given model it is enough to show that the corresponding probability of having $n$ distinct sites 
simultaneously open is smaller than $\rho^{n}$ for some $0\le\rho<(3^d-1)^{-1}$ for $n$ large enough.
From this we can get that the expected number of open paths of length $n$ starting from the origin (which is at most
$(\rho(3^d-1))^n$) tends to $0$ and hence by Markov's inequality, we get that almost surely there is no infinite path i.e, no percolation.  
\end{rem}

\begin{rem}[Peierls' argument for percolation]
\label{rem:perc-standard-arguments}
Recall that the number of  contours surrounding
the origin  in $\mL^{*d}(r)$\footnote{A contour  surrounding
the origin in $\mL^{*d}(r)$ is a
minimal collection of vertices of $\mL^{*d}(r)$ such
that any infinite path on this graph
from the origin has to contain one of these
vertices.} is at most $n(3^d-2)^{n-1}$~\footnote{The bounds $n(3^d-2)^{n-1}$ and $(3^d-1)^n$ (in Remark \ref{rem:non_perc_standard_arguments}) are not tight;
  we use them for simplicity of exposition. For more about the former bound, refer \cite{Lebowitz98,Balister07}.}.
Hence, in order to prove percolation of a
given model using Peierls argument (cf.~\cite[ pp.~17--18]{Grimmett99}), it is enough to show that
the corresponding probability of having $n$ distinct sites simultaneously closed is
smaller than $\rho^{n}$ for some $0\le\rho<(3^d-2)^{-1}$ for $n$ large enough. Thus the expected
the number of closed contours around the origin (which is at most $\sum_{n \geq 1}n(3^d-2)^{n-1}\rho^{n}$) is finite and hence by a duality argument, we can infer that almost surely there will be at least one infinite path i.e, percolation.
\end{rem}

\subsection{$k$-percolation in Boolean model}
\label{sec:k_perc}
Recall that $k$-percolation has already been introduced rigorously in
the introduction itself; see Definition~\ref{d.k-percol}.
The aim of this section is to show that for weakly sub-Poisson point processes, the critical
intensity for $k$-percolation of the Boolean model is non-degenerate. 
\begin{thm}
\label{thm:k_perc_sub-Poisson_pp}
Let $\Phi$ be a simple, stationary point process of intensity~$\lambda$. 
For $k \geq 1$, there exist constants $c(\lam)$ and
$c(\lam,k)$ (not depending on the distribution of $\Phi$) such that
$0 < c(\lam) \leq r_c^1(\Phi)$ provided  $\Phi$ is $\alpha$-weakly
sub-Poisson and
$r_c^k(\Phi) \leq c(\lam,k) < \infty$ provided
$\Phi$ is $\nu$-weakly
sub-Poisson.  
\end{thm}
More simply, Theorem~\ref{thm:k_perc_sub-Poisson_pp} 
gives an upper and lower bound for the critical radius of a stationary weakly sub-Poisson point
process dependent only on its intensity and not on the finer structure.
This is the content of Corollary \ref{cor:non_trivial_k_perc} stated in the
introduction. Recall that examples of such point processes have been already mentioned in Section \ref{ss:sim_ex}.

\begin{rem}[Clique percolation]
\label{rem:cech_complex}
 The \u{C}ech simplicial complex on a point process $\Phi$ is defined
 as the simplicial complex whose $k$-dimensional faces are subsets
 $\{X_0,\ldots,X_k\} \subset \Phi$ such that
 $\bigcap_{i=0}^kB_{X_i}(r) \neq \emptyset.$ Define a graph on the
 $k$-dimensional faces by placing edges between two $k$-dimensional
 faces if they are contained within the same $(k+1)$-dimensional
 face. The case $k=1$ corresponds to the random geometric graph or the
 Boolean model $C(\Phi,r)$. Is there percolation in this graph for any
 $k \geq 1$? This question was posed in \cite[Section 4]{Kahle11} for
 the case of the Poisson point process. The question was motivated by
 positive answer to the discrete analogue of this question for
 Erd\"{o}s-R\'{e}nyi random graphs in \cite{Bollobas09} where it was
 called as {\em clique percolation}. Once we observe that percolation
 of $k$-faces in the \u{C}ech simplicial complex is equivalent to
 $k$-percolation in the Boolean model,  Corollary \ref{cor:non_trivial_k_perc}  answers the question in affirmative not only for the Poisson point process    
but also for weakly sub-Poisson point processeses. 
\end{rem}

\begin{proof}[Proof of Theorem~\ref{thm:k_perc_sub-Poisson_pp}]
As explained before, we shall use the standard arguments as described in Remarks \ref{rem:non_perc_standard_arguments} and \ref{rem:perc-standard-arguments} to show the lower and upper bounds respectively. 

In order to prove the first statement, let $\Phi$ be $\alpha$-weakly sub-Poisson and $r > 0$. Consider the close packed lattice $\mL^{*d}(2r)$. Define the response function $l_r(x,y)
:= \1[x \in Q_r(y)]$ and the corresponding shot-noise field
$V^r_{\Phi}(z)$ on $\mL^{*d}(2r)$. Note that if $C(\Phi,r)$ percolates
then $\mZ^d_{2r}(V^r_{\Phi},\geq 1)$ percolates as well.
We shall now show that there exists a $r >0$ such that
$\mZ^d_{2r}(V^r_{\Phi},\geq 1)$ does not percolate.
For any $n$ and $z_i \in r\mZ^d, 1 \leq i \leq n$, $\sum_{i=1}^n l_r(x,z_i) =  1$ iff $x \in \bigcup_{i=1}^n Q_r(z_i)$ and else $0$. Thus, from Proposition \ref{prop:bounds_shot-noise}, we have that
\begin{eqnarray}
\pr{V^r_{\Phi}(z_i) \geq 1 , 1 \leq i \leq n} & \leq & e^{-sn}
\exp\left\{\lam \int_{\mR^d}(e^{s\sum_{i=1}^nl_r(x,z_i)}-1) \md
  x\right\}\, ,
  \non \\
& = & e^{-sn} \exp\left\{\lam \|\bigcup_{i=1}^n Q_r(z_i)\| (e^s-1)
\right\} , \non \\
\label{eqn:gen_bd_non_perc} & = & (\exp\{-(s+ (1-e^s)\lam (2r)^d) \})^n,
\end{eqnarray}
where $\|\cdot\|$ denote the $d$-dimensional Lebesgue's measure. 
Choosing $s$ large enough that $e^{-s} < (3^d -1)^{-1}$ and then by
continuity of $(s+ (1-e^s)\lam (2r)^d)$ in $r$, we can choose a
$c(\lam,s) > 0$ such that for all $r < c(\lam,s)$, $\exp\{-(s+ (1-e^s)\lam
(2r)^d))\} < (3^d -1)^{-1}$. Now, using the standard argument
involving the expected number of open paths (cf Remark~\ref{rem:non_perc_standard_arguments}),
we can show non-percolation of $\mZ^d_{2r}(V^r_{\Phi},\geq 1)$ for $r
< c(\lam) := \sup_{s > \log(3^d-1)}c(\lam,s)$. Hence for all $r < c(\lam)$,  $C(\Phi,r)$ does not
percolate and so $c(\lam) \leq r_c(\Phi)$.

For the second statement, let $\Phi$ be $\nu$-weakly sub-Poisson.
Consider the close packed lattice, $\mL^{*d}(\frac{r}{\sqrt{d}})$. Define the response function $l_r(x,y) := \1[x \in Q_{\frac{r}{2\sqrt{d}}}(y)]$ and the corresponding
additive shot-noise field $V^r_{\Phi}(z)$ on
$\mL^{*d}(\frac{r}{\sqrt{d}})$. Note that $C_k(\Phi,r)$ percolates if
$\mZ^d_{\frac{r}{\sqrt{d}}}(V^r_{\Phi},\geq \lceil k/2 \rceil)$
percolates,
where $\lceil a \rceil=\min\{z\in\mZ: z\ge a\}$.
We shall now show that there exists a $r < \infty$ such that $\mZ^d_{\frac{r}{\sqrt{d}}}(V^r_{\Phi},\geq \lceil k/2 \rceil)$ percolates.
For any $n$ and $z_i, 1 \leq i \leq n$, from Proposition~\ref{prop:bounds_shot-noise}, we have that
\begin{eqnarray}
\lefteqn{\pr{V^r_{\Phi}(z_i) \leq \lceil k/2 \rceil - 1 , 1 \leq i
    \leq n}}\non \\
 & \leq & e^{sn(\lceil k/2 \rceil - 1)} \exp\left\{\lam \int_{\mR^d}(e^{-s\sum_{i=1}^nl_r(x,z_i)}-1)\, \md x \right\}\ \non \\
& = & e^{sn(\lceil k/2 \rceil - 1)} \exp\left\{\lam \|\bigcup_{i=1}^n Q_{\frac{r}{2\sqrt{d}}}(z_i)\| (e^{-s}-1)  \right\} \non \\
\label{eqn:gen_bd_perc} & = & (\exp\{-((1-e^{-s})\lam (\frac{r}{\sqrt{d}})^d - s(\lceil k/2 \rceil - 1)) \})^n.
\end{eqnarray}
For any $s$, there exists $c(\lam,k,s) < \infty$ such that for all $r
> c(\lam,k,s)$, the last term in the above equation is strictly less
than $(3^d-1)^{-n}$. Thus one can use the standard argument involving the expected number of closed contours
around the origin (cf Remark~\ref{rem:perc-standard-arguments})
to show that $\mZ^d_{\frac{r}{\sqrt{d}}}(V^r_{\Phi},\geq \lceil k/2
\rceil )$ percolates. Further defining $c(\lam,k) := \inf_{s > 0} c(\lam,k,s)$, we have that $C_k(\Phi,r)$ percolates for all
$r > c(\lam,k)$. Thus $r^k_c(\Phi) \leq c(\lam,k)$.
\end{proof}
\begin{rem}
\label{rem:two_bounds_again}
Following Remark~\ref{rem:two_bounds}, 
bounds on the critical radii can be obtained 
by comparison to Poisson processes of
two different intensities:
$0 < c(\lam') \leq r_c^1(\Phi)$ provided $\Phi$ 
is simple and its factorial moment measures 
are bounded by those of a homogeneous Poisson point
process of intensity $\lambda'>0$ and 
$r_c^k(\Phi) \leq c(\lam'',k) < \infty$ provided the void probabilities
of $\Phi$ are bounded by those of the homogeneous Poisson point process of intensity 
$\lambda''>0$. 
\end{rem}

For $k=1$; i.e., for the usual percolation in Boolean model,
we can avoid the usage of exponential estimates of
Proposition~\ref{prop:bounds_shot-noise}  and work with void probabilities and factorial
moment measures only. The gain is improved bounds on the critical radius. 
\begin{thm}
\label{thm:Peierls_arg_bd_pp}
Let $\Phi$ be a stationary point process of intensity $\lambda$ and $\nu$-weakly
sub-Poisson. Then $r_c(\Phi) \leq \tilde{c}(\lam) := \sqrt{d}\left(\frac{\log (3^d-2)}{\lam}\right)^{1/d} \leq c(\lam,1) < \infty$.
\end{thm}
\begin{proof}
As in the second part of the proof of Theorem
\ref{thm:k_perc_sub-Poisson_pp}, consider the close packed lattice
$\mL^{*d}(\frac{r}{\sqrt{d}})$. Define the response function $l_r(x,y)
:= \1[x \in Q_{\frac{r}{2\sqrt{d}}}(y)]$ and the corresponding
extremal shot-noise field $U^r_{\Phi}(z):=\sup_{X\in\Phi}l_r(z,X)$ on
$\mL^{*d}(\frac{r}{\sqrt{d}})$. Now,
note that $C(\Phi,r)$ percolates if
$\{z:U^r_{\Phi}(z) \geq 1\}$ percolates on $\mL^{*d}(\frac{r}{\sqrt{d}})$. We shall
now show that this holds true for $r > \tilde{c}(\lam)$.
Using the ordering of void probabilities we have
\begin{eqnarray}
\pr{U^r_{\Phi}(z_i) = 0 , 1 \leq i \leq n} & = & \pr{\Phi \cap \bigcup_{i=1}^n Q_{\frac{r}{2\sqrt{d}}}(z_i) = \emptyset} \non \\
& \leq & \pr{\Phi_{\lam} \cap \bigcup_{i=1}^n Q_{\frac{r}{2\sqrt{d}}}(z_i) = \emptyset } \non \\
\label{eqn:Peierl_bd_perc} & = & \left(\exp\left\{-\lam
    (\frac{r}{\sqrt{d}})^d \right\}\right)^n\,.
\end{eqnarray}
Clearly, for $r > \tilde{c}(\lam),$ the exponential term above is
less than $(3^d-2)^{-1}$ and thus   $\{z:U^r_{\Phi}(z) \geq 1\}$
percolates by Peierls argument (cf
Remark~\ref{rem:perc-standard-arguments}).
 It is easy to see that for any $s >
0$,  $\exp\{-\lam (\frac{r}{\sqrt{d}})^d\} \leq \exp\{-(1-e^{-s})\lam
(\frac{r}{\sqrt{d}})^d\}$ and hence $\tilde{c}(\lam) \leq c(\lam,1)$.
\end{proof}

Combining the results of Theorem~\ref{thm:Peierls_arg_bd_pp} and \cite[Proposition 6.1]{dcx-clust}, we have the following phase-transition result for usual continuum percolation.
\begin{cor}
\label{cor:phase-transition}
For a stationary weakly sub-Poisson point process $\Phi$, we have that
$0< \frac{1}{(\theta_d)^{1/d}}\le r_c(\Phi)\le
\sqrt{d}\left(\frac{\log (3^d-2)}{
\lambda}\right)^{1/d}<\infty$
\end{cor}

\subsection{Percolation in SINR graphs}
\label{sec:perc_SINR}
Study of percolation in the Boolean model $C(\Phi,r)$ was proposed in~\cite{Gilbert61}
to address the feasibility of multi-hop communications in large
``ad-hoc'' networks, where full connectivity is typically hard to maintain. 
The Signal-to-interference-and-noise ratio (SINR)
model~(see~\cite{sinr_coverage,Dousse_etal_TON,Dousse_etal06}~\footnote{The name 
{\em shot-noise germ-grain process} was also suggested by D.~Stoyan in 
his private communication to BB.})
 is more adequate than the Boolean model
in the context of  wireless communication networks as it  allows
one to take into account the {\em interference} intrinsically related to wireless communications.
For more motivation to study SINR model, refer \cite{subpoisson} and the references therein.

We begin with a formal introduction of the SINR graph model.
In this subsection, we shall work only in $\mR^2$. The parameters of the model are non-negative numbers $P$(signal power), $N$(environmental noise), $\gamma$, $T$(SINR threshold) and an attenuation function $\ell :\mR^2_+ \to \mR_+$  satisfying the following assumptions:
$\ell(x,y)  =  l(|x-y|)$ for some continuous function $l : \mR_+ \to
\mR_+$, strictly decreasing on its support, with  $l(0) \geq TN/P$, $l(\cdot) \leq  1$,
and  $\int_0^{\infty} x l(x) \md x  <  \infty$.
These are exactly the assumptions made in~\cite{Dousse_etal06} and we
refer to this paper for a discussion on their validity.

Given a point process $\Phi$, the {\em interference} generated due to
the point process at a location $x$ is defined as the following
shot-noise field
\label{eqn:interference}
$I_{\Phi}(x) := \sum_{X \in \Phi\setminus\{x\}} l(|X-x|)$.
Define the SINR value as follows :
\begin{equation}
\label{eqn:sinr_defn}
\text{SINR}(x,y,\Phi,\gamma) :=  \frac{Pl(|x-y|)}{N + \gamma P I_{\Phi\setminus\{x\}}(y)}.
\end{equation}

Let $\Phi_B$ and $\Phi_I$ be two point processes. Let $P,N,T > 0$ and $\gamma \geq 0$. The SINR graph is defined as $G(\Phi_B,\Phi_I,\gamma) := (\Phi_B,E(\Phi_B,\Phi_I,\gamma))$ where 
$$E(\Phi_B,\Phi_I,\gamma)\!:=\! \{ \la X,Y \ra\!\in\!\Phi_B^2: \text{SINR}(Y,X,\Phi_I,\gamma)\!>\!T,\, \text{SINR}(X,Y,\Phi_I,\gamma) \!> \!T \}.$$
The SNR graph (i.e, the graph without interference, $\gamma = 0$) is
defined as $G(\Phi_B) := (\Phi_B,E(\Phi_B))$ where 
$$E(\Phi_B) := \{ \la X,Y \ra~\in~\Phi_B^2 : \text{SINR}(X,Y,\Phi_B,0) > T \}.$$

Observe that the SNR graph $G(\Phi)$ is same as the graph $C(\Phi,r_l)$ with $2r_l = l^{-1}(\frac{TN}{P})$. Also, when $\Phi_I = \emptyset$, we shall omit it from the parameters of the SINR graph. Recall that percolation in the above graphs is existence of an infinite connected component in the graph-theoretic sense.

\subsubsection{Poissonian back-bone nodes}
\label{sec:poisson_sinr}

Firstly, we consider the case when the backbone nodes ($\Phi_B$) form
a Poisson point process and in the next section, we shall relax this
assumption. When $\Phi_B = \Phi_{\lam}$, the Poisson point process of intensity
$\lambda$, we shall use $G(\lam,\Phi_I,\gamma)$ and $G(\lam)$ to
denote the SINR and SNR graphs respectively. Denote by
$\lam_c(r):=\lambda (r_c(\Phi_\lambda)/r)^2$
the {\em critical intensity} for percolation of the Boolean model $C(\Phi_\lam,r)$. The
following result guarantees the existence of a $\gamma > 0$ such that
for any sub-Poisson point process $\Phi = \Phi_I$, $G(\lam,\Phi,\gamma)$ will
percolate provided $G(\lam)$ percolates i.e, the SINR graph percolates
for small interference values when the corresponding SNR graph
percolates.
\begin{thm}
\label{thm:sinr_poisson_perc}
Let $\lam > \lam_c(r_l)$ and $\Phi$ be an $\alpha-$ weakly sub-Poisson with mean measure $\mu\|.\|$ for some $\mu > 0$. Then there exists a $\gamma > 0$ such that $G(\lam,\Phi,\gamma)$ percolates.
\end{thm}
Note that we have not assumed the independence of $\Phi$ and
$\Phi_{\lam}$. In particular, $\Phi$ could be $\Phi_{\lam} \cup
\Phi_0$ where $\Phi_0$ is an independent $\alpha-$ weakly sub-Poisson point process. The case
$\Phi_0 = \emptyset$ was proved in \cite{Dousse_etal06}. Our proof
follows their idea of coupling the continuum model with a discrete
model and then using the Peierls argument (see Remark \ref{rem:perc-standard-arguments}). As in
\cite{Dousse_etal06}, it is clear that for  $N \equiv 0$, the
above result holds with $\lam_c(r_l) = 0$.
\begin{proof}[Sketch of the proof of Theorem~\ref{thm:sinr_poisson_perc}]
Our proof follows the arguments  given in \cite{Dousse_etal06} and here,
we will only give a sketch of the proof. The details can be found
in~\cite[Section~6.3.4]{Yogesh_thesis}.

Assuming $\lambda>\lambda_c(\rho_l)$, one observes first
that the graph $G(\lambda)$ also percolates
with any slightly larger constant noise
$N'=N+\delta'$, for some $\delta'>0$.
Essential to the proof of the result
is to show that the
level-set $\{x:I_{\Phi_I}(x)\le M\}$ of the interference
field  percolates (contains an infinite connected component) for sufficiently
large $M$. Suppose that it is true.
Then taking $\gamma=\delta'/M$ one has
percolation of the level-set $\{y:\gamma I_{\Phi_I}(y)\le \delta'\}$.
The main difficulty consists in showing that
$G(\lambda)$ with noise $N'=N+\delta'$
percolates {\em within} an infinite connected component of
$\{y:I_{\Phi_I}(y)\le \delta'\}$. This was done in~\cite{Dousse_etal06},
by mapping both models $G(\lambda)$ and the level-set of the
interference field to a discrete lattice
and showing that both discrete approximations not only percolate
but actually satisfy a stronger condition, related to the Peierls argument.
We follow exactly the same steps and the only fact that we have to
prove,  regarding  the interference, is that there exists a constant
$\epsilon<1$ such that for arbitrary $n\ge1$ and
arbitrary choice of locations $x_1,\ldots,x_n$ one has
$\pr{I_{\Phi_I}(x_i)> M, \, i=1,\ldots, n}\le \epsilon^n$.
In this regard, we use the first statement of
Proposition~\ref{prop:bounds_shot-noise} 
to prove, exactly as in~\cite[Prop. 2]{Dousse_etal06}, that for sufficiently small $s$
it is not larger than $K^n$ for some
constant $K$  which depends on $\lam$ but not on $M$.
This completes the proof.
\end{proof}

\subsubsection{Non-Poissonian back-bone nodes}
\label{sec:non_poisson_sinr}

We shall now consider the case when the backbone nodes are formed by a
$\gamma$-weakly sub-Poisson point process. In this case, we can give a weaker result, namely that
with an increased signal power (i.e, possibly much greater than the critical
power), the SINR graph will percolate for small interference
parameter $\gamma > 0$.
\begin{thm}
\label{thm:perc_sinr_sub-Poisson}
Let $\Phi$ be a stationary, $\nu$-weakly sub-Poisson point process and $\Phi_I$ be an $\alpha-$ weakly sub-Poisson point process with intensity $\mu$ for some $\mu > 0.$ Also assume that $l(x)
> 0$ for all $x \in \mR_+$. Then there exist $P, \gamma > 0$ such that $G(\Phi,\Phi_I,\gamma)$ percolates.
\end{thm}
As in Theorem \ref{thm:sinr_poisson_perc}, we have not assumed the independence of $\Phi_I$ and $\Phi$. For example, $\Phi_I = \Phi \cup \Phi_0$ where $\Phi$ and $\Phi_0$ are independent $\nu$-weakly and $\alpha-$ weakly sub-Poisson point processes respectively. Let us also justify the assumption of unbounded support for $l(\cdot)$. Suppose that $r = \sup \{x : l(x) > 0 \} < \infty$. Then if $C(\Phi,r)$ is sub-critical, $G(\Phi,\Phi_I,\gamma)$ will be sub-critical for any $\Phi_I,P,\gamma$.
\begin{proof}[Sketch of the proof of Theorem~\ref{thm:perc_sinr_sub-Poisson}]
In this scenario, increased power is equivalent to increased radius in
the Boolean model corresponding to SNR model. From this
observation, it follows from Theorem~\ref{thm:Peierls_arg_bd_pp} that with
possibly increased power the associated SNR model percolates. Then, we
use the approach from the proof of
Theorem~\ref{thm:sinr_poisson_perc} to obtain
a $\gamma > 0$ such that the SINR network percolates as well.
The details can be found in \cite[Section~6.3.4]{Yogesh_thesis}.
\end{proof}
For further  discussion on $dcx$ ordering in the context of
communication networks see~\cite{subpoisson}.

\section{Super-Poisson point process with null critical radius}
\label{sec:ex}
The objective of this section is to show an example of highly clustered 
and well percolating point process. More precisely we construct 
a Poisson-Poisson cluster point process 
(which is known to be $dcx$ larger than the Poisson point process) for which
$r_c = 0$. This invalidates the temptation to conjecture the monotonicity of 
$r_c$ with respect to the $dcx$ order (and hence with respect to void
probabilities and moment measures) of point process, in full generality.

\begin{exe}[Poisson-Poisson cluster point process with annular clusters]
\label{ex.PoPoClust} 
Let $\Phi_\alpha$ be the Poisson point process of intensity $\alpha$ on the plane
$\mR^2$; we call it the process of cluster centers.
For any $\delta,R,\mu$ such that $0<\delta\le R<\infty$ and
$0<\mu<\infty$, consider a Poisson-Poisson cluster point process  $\Phi_\alpha^{R,\delta,\mu}$; 
i.e., a Cox point process  with the random intensity measure 
$\Lam(\cdot) := \mu \sum_{X \in \Phi_{\alpha}} \cX(x,\cdot-x)$, where
 $\cX(x,\cdot)$ is  the uniform distribution on the annulus
$B_O(R)\setminus B_O(R-\delta)$ centered at $x$ of inner and
outer radii $R-\delta$ and $R$ respectively; see Figure~\ref{f.PoPoClust}.

By~\cite[Proposition~5.2]{snorder}, it is 
a super-Poisson point process.  More precisely,
$\Phi_{\lambda}\le_{dcx}\Phi_\alpha^{R,\delta,\mu}$, where
$\Phi_\lambda$ is homogeneous Poisson point process of intensity $\lambda=\alpha\mu$.
\end{exe}

\begin{figure}[!t]
\begin{center}
\includegraphics[width=1.\linewidth]{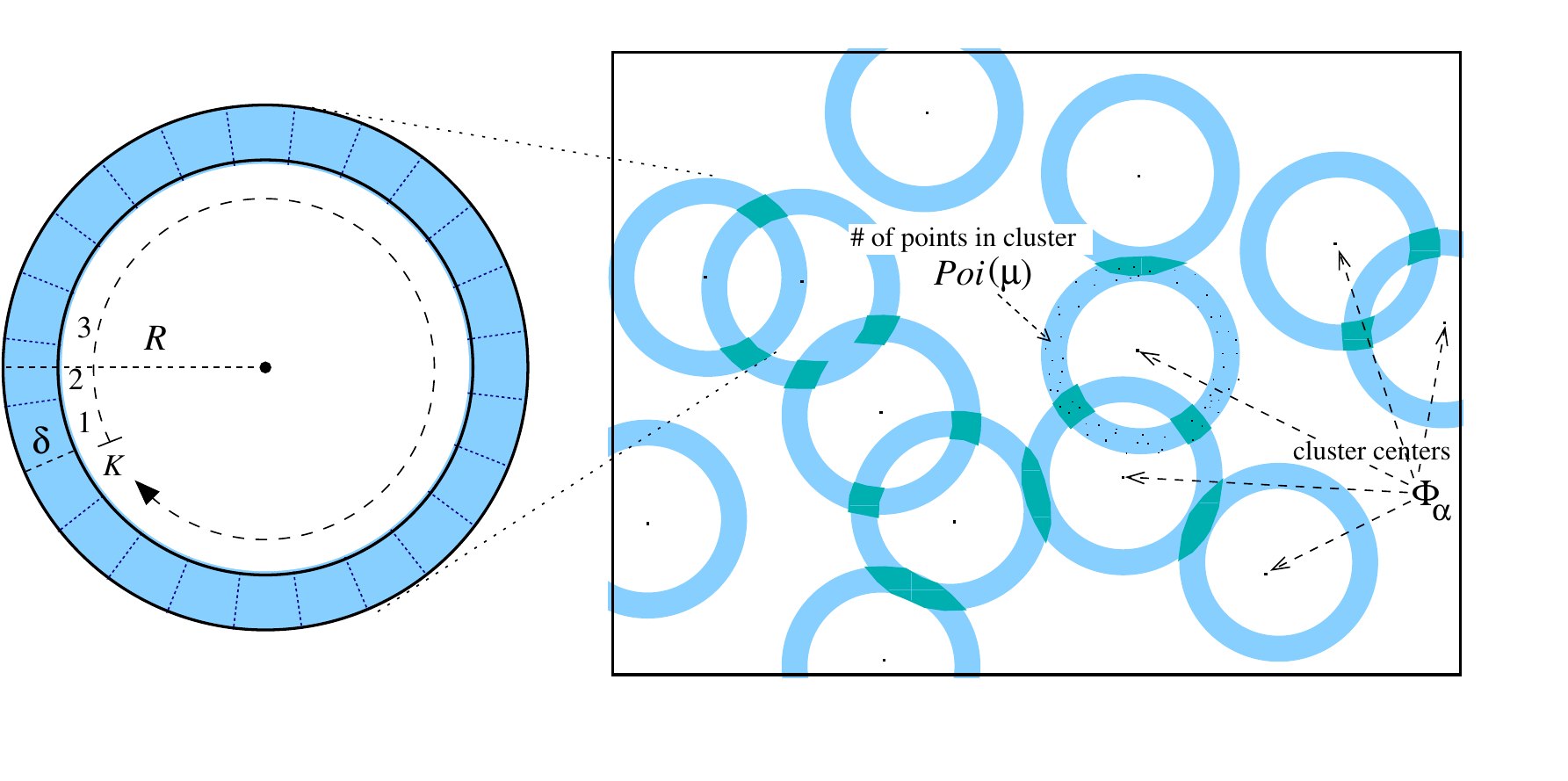}
\end{center}
\vspace{-7ex}
\caption{\label{f.PoPoClust}
Poisson-Poisson cluster process of annular cluster;
cf. Example~\ref{ex.PoPoClust}.}
  \end{figure}

For a given arbitrarily large intensity $\lambda<\infty$, 
taking sufficiently small $\alpha,R$, $\delta=R$ and sufficiently large $\mu$, it
is straightforward to construct a Poisson-Poisson cluster point process
$\Phi_\alpha^{R,R,\mu}$ with spherical clusters, 
which has an arbitrarily large critical radius $r_c$ for percolation.
It is less evident that one can construct a Poisson-Poisson cluster point process
that always percolates, i.e., with degenerate critical radius $r_c=0$.

\begin{prop}
\label{p.PopPoClust}
Let $\Phi_\alpha^{R,\delta,\mu}$ be a Poisson-Poisson cluster point process with 
annular clusters on the plane~$\mR^2$ as in
Example~\ref{ex.PoPoClust}. Given arbitrarily small $a,r>0$, there
exist constants $\alpha,\mu,\delta,R$ such that $0<\alpha,\mu,\delta,R<\infty,$ the intensity $\alpha\mu$ of
$\Phi_\alpha^{R,\delta,\mu}$ is equal to $a$ and 
the critical radius for percolation $r_c(\Phi_\alpha^{R,\delta,\mu})\le r$.
Moreover, for any $a>0$ there exists point process $\Phi$
of intensity $a$, which is $dcx$-larger than the Poisson point process of intensity
$a$, and which percolates for any $r>0$; i.e., $r_c(\Phi)=0$.
\end{prop}

\begin{proof}
Let $a,r>0$ be given. Assume $\delta=r/2$. 
We will show that there exist sufficiently
large $\mu,R$ such that $r_c(\Phi_\alpha^{R,\delta,\mu})\le r$ where $\alpha= a/\mu$.
In this regard, denote $K:=2\pi R/r$ and 
assume that $R$ is chosen such that $K$ is an
integer. For a $\alpha > 0$ and any point (cluster center) $X_i\in\Phi_\alpha,$ let us partition
the annular support  $A_{X_i}(R,\delta):=B_{X_i}(R)\setminus
B_{X_i}(R-\delta)$  of the translation kernel $X_i + \cX(X_i,\cdot)$ 
(support of the Poisson point process constituting the  cluster centered at $X_i$)
into~$K$ cells as shown in
Figure~\ref{f.PoPoClust}. We will call  $X_i$
``open'' if in each of the $K$ cells of $A_{X_i}(R,\delta),$
there exists at least one replication of the point
$X_i$ among the Poisson $Poi(\mu)$ (with $\alpha= a/\mu$) number of total replications of the
point $X_i$.
 Note that given $\Phi_\alpha$, each point
$X_i\in\Phi_\alpha$ is open with probability
$p(R,\mu):=(1-e^{-\mu/K})^K$, independently of 
other points of $\Phi_\alpha$. Consequently, open points of
$\Phi_\alpha$ form a Poisson
point process of intensity $\alpha p(R,\mu)$; call it $\Phi_{open}$.
Note that the maximal distance between any
two points in two neighbouring cells of the same cluster is not
larger than $2(\delta+2\pi R/K)=2r$. Similarly,  
 the maximal distance between any
two points in two non-disjoint cells of two different clusters is not
larger than $2(\delta+2\pi R/K)=2r$. Consequently, if the Boolean
model $C(\Phi_{open},A_{0}(R,\delta))$ with annular grains 
percolates then the Boolean model 
$C(\Phi_\alpha^{R,\delta,\mu}, r)$ with spherical grains of radius $r$
percolates as well. The former Boolean model percolates if and only
if $C(\Phi_{open},B_{0}(R))$ percolates. Hence, in order to guarantee 
 $r_c(\Phi_\alpha^{R,\delta,\mu})\le r,$ it is enough to chose
$R,\mu$ such that the volume fraction
$1 - e^{-\alpha p(R,\mu)\pi R^2} = 1 - e^{-a p(R,\mu)\pi R^2/\mu}$  
is larger than the critical volume fraction
for the percolation of the spherical Boolean model on the plane.
In what follows, we will show that by choosing appropriate 
$R,\mu$ one can make  $p(R,\mu) R^2/\mu$ arbitrarily large. Indeed, 
take 
$$\mu:=\mu(R)=\frac{2\pi R}{r}\log{\frac R{\sqrt{\log R}}}=
\frac{2\pi R}{r}\Bigl(\log R- \frac12\log\log R\Bigr)\,.
$$ 
Then, as $R\to\infty$
\begin{eqnarray*}
p(R,\mu) R^2/\mu
&=&\frac{R^2}{\mu}(1-e^{-\mu r/(2\pi R)})^{2\pi R/r}\\
&=&\frac{Rr}{2\pi(\log R- \frac12\log\log R)}
\Bigl(1-\frac{\sqrt{\log R}}{R}\Bigr)^{2\pi R/r}\\
&=&
e^{O(1)+\log R -\log(2\pi(\log R- \frac12\log\log R))-O(1)\sqrt{\log R}}\to\infty\,.
\end{eqnarray*}
This completes the proof of the first statement.

In order to prove
the second statement, for a given $a>0$, denote $a_n:=a/2^n$ and let
$r_n=1/n$. Consider a sequence of independent (super-Poisson)
Poisson-Poisson cluster 
point processes $\Phi_n=\Phi_{\alpha_n}^{R_n,\delta_n,\mu_n}$ with  
intensities $\lambda_n:=\alpha_n\mu_n= a_n$, satisfying 
$r_c(\Phi_n)\le r_n$.  The existence of such point processes was
shown in the first part of the proof.
By the fact that $\Phi_n$ are super-Poisson for all $n\ge0$
and by~\cite[Proposition 3.2(4)]{snorder} the superposition
$\Phi=\bigcup_{n=1}^\infty \Phi_n$ is $dcx$-larger than Poisson point process
of intensity~$a$. Obviously $r_c(\Phi)=0$. This completes the proof of
the second statement.
\end{proof}

\begin{rem}
By Proposition~\ref{p.PopPoClust}, we know that
there exists point process $\Phi$ with intensity $a>0$ such that
$r_c(\Phi)=0$ and $\Phi_a\le_{dcx}\Phi$, where $\Phi_a$ is homogeneous
Poisson point process.  Since one knows that  $r_c(\Phi_a)>0$ so 
$\Phi$  is a counterexample to the  monotonicity of 
$r_c$ in  $dcx$ ordering of point processes.
\end{rem}

\section{Concluding remarks}
\label{sec:conclusion}
We come back to the initial heuristic discussed in the Introduction ---  
clustering in a point process should increase the  critical radius 
for the percolation of the corresponding continuum percolation model.
As we have seen, even a relatively strong tool such as the $dcx$ order
falls short, when it comes to making a formal statement of this heuristic. 

The two natural questions are what would be a more suitable measure
of clustering that can be used to affirm the heuristic and
whether $dcx$ order can satisfy a weaker version of the conjecture.

As regards the first question, one might start by looking at other
dependence orders such as super-modular, component-wise convex or
convex order but it has been already shown that the first two are not suited to comparison of clustering
in point processes (cf.~\cite[Section~4.4]{Yogesh_thesis}). 
Properties of convex order on point processes are yet to be investigated fully and this research direction is interesting in its own right, apart from its relation to the above conjecture. In a similar vein, it is of
potential interest to study other stochastic orders on point processes.

On the second question, it is pertinent to note that sub-Poisson point processes surprisingly exhibited non-trivial phase transitions for percolation. 
Such well-behavedness of the sub-Poisson point processes
makes us wonder if it is possible to prove a rephrased conjecture
saying that any homogeneous sub-Poisson point process has a smaller critical radius for
percolation than the Poisson point process of the same intensity.
Such a conjecture matches well
with~\cite[Conjecture~4.6]{BenjaminiStauffer2011}. 

\vspace{2ex}

\end{document}